\documentclass[11pt,letterpaper]{amsart}
\usepackage[english]{babel}
\usepackage{amscd}
\usepackage[all,cmtip]{xy} 
\usepackage{amsmath,amsthm,amssymb,amsfonts}
\usepackage[bookmarksopen=true]{hyperref}
\usepackage[usenames, dvipsnames]{color}
\usepackage{graphicx}
\usepackage{hyperref}

\makeatletter
\let\@fnsymbol\@arabic
\makeatother

\theoremstyle{definition}
\newtheorem{definition}{Definition}[section]

\newtheorem{remark}[definition]{Remark}

\theoremstyle{break}
\newtheorem{lemma}[definition]{Lemma}
\newtheorem{proposition}[definition]{Proposition}

\newtheorem{theorem}[definition]{Theorem}
\newtheorem{corollary}[definition]{Corollary}

\newtheorem*{maintheorem*}{MainTheorem(cf. Theorem\til\ref{main})}
\newtheorem*{theorem*}{Theorem}
\newtheorem*{remark*}{Remark}

\def\til{~}

\def\N{\mathbb N}

\def\Z{\mathbb Z}

\def\mm{\mathfrak m}  

\def\pp{\mathfrak p}
\def\qq{\mathfrak q}
\def\aa{\mathfrak a}
\def\bb{\mathfrak b}
\def\cc{\mathfrak c}
\def\H{H} 

\def\la{\longrightarrow}

\def\:{\colon}

\def\Hom{\mathrm{Hom}}

\def\Min{\mathrm{Min}}

\def\codim{\mathrm{codim}}
\def\altezza{\mathrm{ht}}   

\def\reg{\mathrm{reg}}

\def\Ext{\mathrm{Ext}}
\def\Spec{\mathrm{Spec}}

\def\canoide{\mathrm{\mathit w}}

\def\obj{\mathrm{Ob}}

\def\lk{\mathrm{lk}}

\def\sgn{\mathrm{sgn}}

\def\se{\subseteq}

\def\iso{\cong}

\def\proj{\mathrm{Proj}}

\def\id{\mathrm{id}}

\def\cocoa{\mbox{\rm 
   C\kern-.13em o\kern-.07 em C\kern-.13em o\kern-.15em A}}

\newcommand\Exti[4]{\Ext^{#1} _{#2} ({#3}, {#4})}
\newcommand\Homi[3]{\Hom_{#1} (#2,#3)}

\newcommand\fg[0]{finitely generated{}}

\newcommand\fff[3]{\mathcal {#1}^{#2}({#3})}

\newcommand\fude[1]{\mathbf{#1}}

\begin{document}

\title{Lefschetz duality for local cohomology}
\author{Matteo Varbaro} 
\email{varbaro@dima.unige.it}
\address{Dipartimento di Matematica, Universit\'a di Genova, Italy} 
\author{Hongmiao Yu}
\email{yu@dima.unige.it}
\address{Dipartimento di Matematica, Universit\'a di Genova, Italy}
 \thanks{Both authors are supported by PRIN  2020355BBY ``Squarefree Gr\"obner degenerations, special varieties and related topics".} 
  \date{}
\maketitle

\begin{abstract}
Since the 1974 paper by Peskine and Szpiro, liaison theory via complete intersections, and more generally via Gorenstein varieties, has become a standard tool kit in commutative algebra and algebraic geometry, allowing to compare algebraic features of linked varieties. In this paper we develop a liaison theory via quasi-Gorenstein varieties, a much broader class than Gorenstein varieties. As applications, we derive a connectedness property of quasi-Gorenstein subspace arrangements generalizing previous results by Benedetti and the first author, and we deduce the classical topological Lefschetz duality via the Stanley-Reisner correspondence.
\end{abstract}

\section{Introduction}

The notion of a liaison between (coordinate rings of) algebraic varieties was introduced about fifty years ago from Peskine and Szpiro in \cite{PS}, even if already Max Noether in 1882 was using related concepts \cite{MaxN}. The theory of liaison (or linkage) became later popular with the work of Rao \cite{Rao}, who analysed the relationship between the deficiency modules of two projective curves in $\mathbb{P}^3$ whose union is (or more precisely linked by) an intersection of two surfaces. The basic idea is indeed that linked varieties, though they may be quite disparate, do share many features in common. The classical theory concerns varieties linked via complete intersections, but it was later extended by Schenzel \cite{schenzel} to linkage via the larger class of Gorenstein varieties, which play a key role in duality theory (see also the book of \cite{Mig}). In this paper, we develop a liaison theory via quasi-Gorenstein varieties, which is a much broader class of Gorenstein varieties: for example for any projective variety $X$ with very ample canonical (resp. anti-canonical) bundle, the canonical (resp. anti-canonical) ring of sections is quasi-Gorenstein, but not necessarily Gorenstein (e.g. if $H^i(X,\mathcal{O}_X)\neq 0$ for some $0<i<\dim(X)$); or, the coordinate ring of a projective variety with trivial canonical bundle, whenever embedded in a projective space in a projectively normal way, is quasi-Gorenstein, but not necessarily Gorenstein; or still, the Stanley-Reisner ring of any triangulated homology manifold is quasi-Gorenstein, while only triangulated homology spheres yield Gorenstein rings. Later we show two applications of this theory:
\begin{enumerate}
\item In Section \ref{sec:conn} we study the connectedness properties of projective varieties with a quasi-Gorenstein coordinate ring: in particular, in Theorem \ref{t:conn} we extend the main result of \cite{BV} to quasi-Gorenstein subspace arrangements.
\item In Section \ref{s:Lefschetz} we show how the liaison theory for quasi-Gorenstein varieties can be used to get, via the Stanley-Reisner correspondence, the topological Lefschetz duality.
\end{enumerate}

Precisely, in Section \ref{s:liaison} liaison theory via quasi-Gorenstein rings is developed. Throughout we work with $\N$-graded rings $R$ where the degree zero part is not necessarily a field: in this way we treat at the same time the local theory and the theory for projective varieties over a field. The main result is Theorem \ref{main}, linking in a long exact sequence of graded modules the local cohomology modules of three rings $R/\aa,R/\bb,R/\cc$, such that $R/\cc$  is quasi-Gorenstein, $\aa$ and $\bb$ are linked via $\cc$ and $R/\aa$ is  generalized Cohen-Macaulay. On the way we prove also some structure theorems of quasi-Gorenstein rings, like a Poincar\'e duality for their local cohomology modules (Proposition \ref{cPoincare}) or the fact that for $F$-injective generalized Cohen-Macaulay positively graded algebras $A$ over a field of positive characteristic with $H_{\mm}^{\dim A}(A)_0=0$, Gorenstein and quasi-Gorenstein are equivalent notions (Corollary \ref{qGorsing}). We also prove Corollary \ref{serre}, connecting the Serre properties $(S_{\ell})$ with properties of the local cohomology modules of $R/\bb$ and of $R/\cc$ (independently on the fact that $R/\aa$ is generalized Cohen-Macaulay).

In Section \ref{sec:conn} we derive some applications regarding the connectedness properties of quasi-Gorenstein rings: Theorem \ref{t:connectedness} provides an intriguing result on the connectedness of the dual graph of a reduced quasi-Gorenstein ring with Cohen-Macaulay minimal prime ideals. With more restrictive assumptions we get a more precise conclusion: Theorem \ref{t:conn} shows a strong connectedness feature of quasi-Gorenstein subspace arrangements, generalizing the main result of \cite{BV} about Gorenstein line arrangements.

In Section \ref{s:Lefschetz} we apply our previous results to Stanley-Reisner rings. Theorem \ref{t:qGpseudo} shows that quasi-Gorenstein Stanley-Reisner rings with canonical module generated in degree 0 correspond to normal pseudo-manifolds with non-zero top homology. It was essentially proved by Gr\"abe in \cite{Grabe} that triangulations of orientable homology manifolds that are not homology spheres yield quasi-Gorenstein Stanley-Reisner rings; in Corollary \ref{c:manifold} we establish the converse for Buchsbaum complexes. Finally, in Theorem \ref{Lefschetz duality} we derive the classical topological Lefschetz duality for triangulated homology manifolds (see also the discussion below and Lemma \ref{l:homotopy}).

\bigskip

{\bf Acknowledgments}: Both authors wish to thank Srikanth Iyengar for useful discussions on topics related to the ones discussed in this paper. Also, thanks to the anonymous referee for thoroughly reading the paper and making helpful suggestions to improve the exposition.

\section{Liaison by quasi-Gorenstein ideals}\label{s:liaison}

\noindent {\bf Notation.} Let $R=\oplus_{i\in\N}R_i$ be a graded $n$-dimensional 
Gorenstein ring with a unique homogeneous maximal ideal, say $\mm$. Notice that $R_0$ is necessarily a local ring with maximal ideal $\mm_0=\mm\cap R_0$; let us call $k=R_0/\mm_0$ the residue field of $R_0$, and $E$ the injective hull $E_{R_0}(k)$ (hence $E$ is a graded $R$-module concentrated in degree 0). We will assume that $R_0$ is complete, and for a graded $R$-module $M$ we denote by $M^\vee=\oplus_{i\in\Z}(M^\vee)_i$ the {\it graded Matlis dual} of $M$, namely
\[(M^\vee)_i=\Hom_{R_0}(M_{-i},E).\]
Notice that, even if $M$ is a finitely generated graded $R$-module, in general $M^\vee$ is different from $\Hom_{R_0}(M,E)$, since $M$ may not be finitely generated as an $R_0$-module.

Since $R$ is Gorenstein, there exists an integer $r$ such that $R(r)\iso\canoide_R$ is a graded canonical module. In the special case that $R_0$ is a field, the integer $r$ is called the {\it $a$-invariant} of $R$.
\begin{definition}
Let $\aa$, $\bb$, $\cc$ be homogeneous ideals of $R$. We say that {\it $\aa$ and $\bb$ are linked via the quasi-Gorenstein ideal $\cc$} if:
\begin{enumerate}
\item $\aa$ and $\bb$ are unmixed ideals of height $g$ with $\cc\se\aa\cap\bb$.
\item $\bb=\cc:_R\aa$ and $\aa=\cc:_R\bb$ (hence $\bb/\cc\iso\Hom_R(R/\aa, R/\cc)$ and $\aa/\cc\iso\Hom_R(R/\bb, R/\cc)$).
\item There exists $c\in \Z$ such that $(R/\cc)(c)\iso\Ext^g_R(R/\cc,R(r))$.
\end{enumerate}
\end{definition}

The third condition means that $R/\cc$ is a {\it quasi-Gorenstein ring}. Notice that $R/\cc$ must satisfy Serre's condition $(S_2)$: in particular $R/\cc$ will be Gorenstein as soon it has dimension $\leq 2$. For simplicity and without loss of generality, we will assume throughout that $n-g=d=\dim(R/\cc)\geq 2$. Finally, for each graded $R$-module, we write
$\canoide_{M}$ for $\Exti{n-\dim(M)}R{M}{R(r)}$. 

\begin{remark}
For a homogeneous ideal $I\subset R$ such that $A=R/I$ is Cohen-Macaulay, the $A$-module $\canoide_{A}$ is called in \cite{CMrings} {\it graded canonical module} of $A$. In this case $\omega_A$ is a {\it dualizing complex} for $A$. Often in the literature the $A$-module $\canoide_{A}$ is called canonical module even if $A$ is not Cohen-Macaulay, however in general $\canoide_{A}$ may not be a dualizing complex for $A$.

In our setting $R/\cc$ is Cohen-Macaulay precisely when $R/\cc$ is Gorenstein, and the proofs of the results of Schenzel in \cite{schenzel} that we are going to generalize rely on the fact that $R/\cc(c)$ is a dualizing complex for $R/\cc$, a property that fails in our setting: even if $R/\cc$ admits a dualizing complex under our assumptions, this is not useful to our purposes, so we must change the proof-strategy used in \cite{schenzel}.
\end{remark}

\begin{remark}
It is possible that $R$ is concentrated in degree 0, namely $R_i=0$ for all $i\neq 0$. In this case the words ``graded'' and ``homogeneous''  are vacuous assumptions: homogeneous ideals are just ideals, graded $R$-modules are just $R$-modules, graded maps of graded $R$-modules are just maps of $R$-modules.

A somehow opposite case is when $R_0=k$: In this case $E=k$ and $M^\vee$ is the graded dual as a graded $k$-vector space.
\end{remark}

\begin{lemma}\label{exactseq}
With the above assumptions, we have graded isomorphisms
\[\canoide_{R/\aa}\iso{(\bb/\cc)}(c), \ \ \ \canoide_{R/\bb}\iso{(\aa/\cc)}(c).\]
Furthermore we have a long exact sequence of graded $R$-modules
\begin{eqnarray*}
\cdots \la\H^{d-2}_\mm(\canoide_{R/\aa})\la\H^{d-2}_\mm(R/\cc)(c)\la\H^{d-2}_\mm(R/\bb)(c)\la\H^{d-1}_\mm(\canoide_{R/\aa})\\
\la \H^{d-1}_\mm(R/\cc)(c)\la\H^{d-1}_\mm(R/\bb)(c)\la\H^{d}_\mm(\canoide_{R/\aa})\la (R/\aa)^\vee\la0
\end{eqnarray*}
\end{lemma}
\begin{proof}
We are going to show that $\canoide_{R/\aa}\iso{(\bb/\cc)}(c)$, the other isomorphism is obtained symmetrically. By the assumptions, we have 
\[\bb/\cc\iso\Homi {R}{R/\aa}{R/\cc}\iso\Homi {R/\cc}{R/\aa}{R/\cc}.\] 
We want to show by induction on $g=\altezza(\cc)$ that 
$$\Exti gR{R/\aa}{\canoide_R}\iso\Homi {R/\cc}{R/\aa}{\canoide_{R/\cc}}.$$
\begin{enumerate}
\item[$g=0:$] Since $\cc\se\aa$ and $R/\cc(c)\iso\Ext^g_R(R/\cc,R(r))=\Homi{R}{R/\cc}{R(r)}$, we have 
\begin{eqnarray*}
\Homi {R/\cc}{R/\aa}{\canoide_{R/\cc}}&\iso&\Homi {R/\cc}{R/\aa}{R/\cc(c)}\\
&\iso&\Homi {R/\cc}{R/\aa}{\Homi{R}{R/\cc}{R(r)}}\\
&\iso&\Homi{R}{R/\aa}{R(r)},
\end{eqnarray*}
where the third isomorphism follows because $\cc\subseteq \aa$.
\item[$g>0:$] 
Because $\mathrm{grade}(\cc)=g>0$, the ideal $\cc$ contains an $R$-regular element. We can choose $x\in\cc$ such an $R$-regular element homogeneous of degree $s$. We set $R'=R/(x)$, $\cc'=\cc/(x)$ and $\aa'=\aa/(x)$. Then $R'$ is a graded Gorenstein ring with graded canonical module $\canoide_{R'}\iso R'(r+s)$ (\cite[Corollary 3.6.14]{CMrings}) and $R'/\cc'(c+s)= R/\cc (c+s)\iso \Ext^g_R(R/\cc,R(r))(s)\iso \Ext^{g-1}_{R'}(R'/\cc',R'(r+s))$, where the last isomorphism follows by \cite[Lemma 3.1.16]{CMrings}.
By  the induction hypothesis, using once again \cite[Lemma 3.1.16]{CMrings}, we have:
\begin{eqnarray*}
\Exti gR{R/\aa}{\canoide_R}&\iso&\Exti gR{R/\aa}{R(r)}\\
&\iso&\Exti {g-1}{R'}{R'/\aa'}{R'(r+s)}(-s)\\
&\iso&\Exti {g-1}{R'}{R'/\aa'}{\canoide_{R'}}(-s)\\
&\iso&\Homi {R'/\cc'}{R'/\aa'}{\canoide_{R'/\cc'}}(-s)\\
&\iso&\Homi {R'/\cc'}{R'/\aa'}{R'/\cc'(c+s)}(-s)\\
&\iso&\Homi {R/\cc}{R/\aa}{R/\cc}(c).
\end{eqnarray*}
\end{enumerate}
The existence of the long exact sequence now easily follows: we have a short exact sequence
\[0\la \canoide_{R/\aa}\iso{(\bb/\cc)}(c)\la R/\cc(c)\la R/\bb(c)\la 0.\]
From this we get the long exact sequence of local cohomology modules, that is almost the desired one. The only difference stands in the last piece, that is
\[\cdots\la\H^{d}_\mm(\canoide_{R/\aa})\la\H^{d}_\mm(R/\cc)(c)\la\H^{d}_\mm(R/\bb)(c)\la 0.\]
The point is that, using the local duality theorem for graded $R$-modules (\cite[Theorem 3.6.19]{CMrings}), taking the Matlis dual of the surjection $\H^{d}_\mm(R/\cc)(c)\to\H^{d}_\mm(R/\bb)(c)\to 0$, we get the injection $0\to \canoide_{R/\bb}(-c)\cong \aa/\cc\to R/\cc$, whose cokernel is isomorphic to $R/\aa$. Therefore the kernel of the above surjection is Matlis dual to $R/\aa$, and this lets us conclude.
\end{proof}

\begin{corollary}\label{serre}
With our notation, we have that $R/\aa$ always satisfies $(S_1)$. Furthermore, for any natural number $\ell\geq 2$, the following are equivalent:
\begin{enumerate}
\item $R/\aa$ satisfies Serre's condition $(S_{\ell})$.
\item The natural map $\H^i_\mm(R/\cc)\la\H^i_\mm(R/\bb)$ is an isomorphism for $d-\ell+1< i <d$ and surjective for $i=d-\ell+1$.
\end{enumerate}
\end{corollary}
\begin{proof}
By definition $R/\aa$ has no embedded primes, so it satisfies $(S_1)$. When $\ell\geq 2$, by \cite[Theorem 1.14]{schenzel2} $R/\aa$ satisfies $(S_{\ell})$ if and only if the following two conditions hold:
\begin{itemize}
\item[(a)] The natural map $R/\aa\to \canoide_{\canoide_{R/\aa}}$ is an isomorphism.
\item[(b)] $\H^i_\mm(\canoide_{R/\aa})=0$ for $d-\ell+2\leq i<d$.
\end{itemize}
By local duality theorem for graded $R$-modules (\cite[Theorem 3.6.19]{CMrings}) condition (a) is equivalent to the fact the natural map $\H^{d}_\mm(\canoide_{R/\aa})\la (R/\aa)^\vee$ is an isomorphism. By Lemma \ref{exactseq}, it is surjective in any case (which is indeed equivalent to the fact that $R/\aa$ satisfies $(S_1)$), and it is injective if and only if the natural map $\H^{d-1}_\mm(R/\cc)\la\H^{d-1}_\mm(R/\bb)$ is surjective. Again using Lemma \ref{exactseq}, fixed $2\leq i<d$, the vanishing of $\H^i_\mm(\canoide_{R/\aa})$ is equivalent to the fact that the natural map $\H^i_\mm(R/\cc)\la\H^i_\mm(R/\bb)$ is injective and the natural map $\H^{i-1}_\mm(R/\cc)\la\H^{i-1}_\mm(R/\bb)$ is surjective. Hence conditions (a) and (b) together are equivalent to (2), so we conclude.
\end{proof}

\begin{remark}
The above corollary generalizes \cite[Theorem 4.1]{schenzel}, because if $R/\cc$ is Gorenstein then $\H^i_\mm(R/\cc)=0$ for all $i<d$, so condition (2) is equivalent to say that $\H^i_\mm(R/\bb)=0$ for all $d-\ell+1<i<d$.
\end{remark}

The following remark is well-known, however we give a proof for the convenience of the reader.

\begin{remark}\label{torsionfunc}
There is a natural transformation
$$\eta: \fude R\Gamma_\mm\la \id_{\fff D{}R}$$ on the derived category $\fff D{}R$.
Denoting the full subcategory of $\fff D{}R$ of bounded below complexes by $\fff D{+}R$, if  $X^\bullet\in\obj(\fff D{+}R)$ is such that $\H^i(X^\bullet)$  is an $R$-module of finite length for each $i\in\Z$, then the natural morphism $\eta(X^\bullet): \fude R\Gamma_\mm(X^\bullet)\la X^\bullet$ is an isomorphism in the derived category $\fff D{}R$.
\proof
Since for each $R$-module $N$ there is a natural inclusion $\Gamma_\mm(N)\se N$, there is a natural transformation $\eta':\Gamma_\mm\la  \id_{\fff K{}R} $ on the homotopy category $\fff K{}R$.
Since the category of $R$-modules has enough injectives, each complex $Y^\bullet$ in $\fff D{}R$ admits a quasi-isomorphism to a complex of injective $R$-modules $I_Y^\bullet$ by \cite[I. Corollary $5.3.\gamma$ and Lemma $4.6(2)$]{RaD}. 
Hence $\eta=(\eta_Y)_{Y^\bullet\in\obj(\fff D{}R)}:  \fude R\Gamma_\mm \la \id_{\fff D{}R}$ with $\eta_Y=Q(\eta'_{I_Y})$ is a natural transformation, where $Q: \fff K{}R\la \fff D{}R$ is the localization functor.\\
If $M$ is an $R$-module of finite length, then we have
\begin{eqnarray*} \H^{i}(\fude R\Gamma_\mm(M))\iso\H^i_\mm(M)=
\begin{cases}
M&\text{  if \ } i=0,\\
0&\text{ if \ } i \not=0.
\end{cases}
\end{eqnarray*}
Hence $\eta_M:  \fude R\Gamma_\mm(M)\la M$ is an isomorphism.
Since the category of $R$-modules of finite length is a thick subcategory of the category of $R$-modules and since the functors  $\fude R\Gamma_\mm$, $\id_{\fff D{+}R}$ are way-out right,
by the Lemma on Way-out Functors \cite[Proposition $7.1$ ($ii$)]{RaD} we have that $\eta(X^\bullet): \fude R\Gamma_\mm(X^\bullet)\la X^\bullet$ is an isomorphism for each  complex $X^\bullet\in\obj(\fff D{+}R)$ such that $\H^i(X^\bullet)$  is an $R$-module of finite length for each $i\in\Z$.
\endproof
\end{remark}

We say that a \fg\ graded $R$-module $M$ is {\it generalized Cohen-Macaulay} if the $R$-module $\H^i_\mm(M)$ has finite length for all $i=0,\ldots ,\dim(M)-1$. This is equivalent to ask that $M_\pp$ is Cohen-Macaulay for every homogeneous prime ideal $\pp\neq \mm$. 

\smallskip

The following result already appeared without a formal proof in \cite[Korollar 1.4]{schenzel3}. For the reader's convenience we restate it here giving a complete proof.

\begin{lemma}\label{gr:Schenzel}
Let $M$ be a graded generalized Cohen-Macaulay finitely generated $R$-module of dimension $d$. Then there are isomorphisms of graded $R$-modules
\[H_\mm^{d-i}(M)^\vee\iso\H_\mm^{i+1}(\canoide_M)\]
for all $0<i<d-1$. Furthermore, if $d\geq 2$, there is the short exact sequence of graded $R$-modules:
\[0\la \H_\mm^1(M)^\vee\la H_\mm^d(\canoide_M)\la M^\vee\la \H_\mm^0(M)^\vee\la 0\]
\end{lemma}
\begin{proof}
Let $J^\bullet$ be an object in the bounded derived category ${\bf D}^b(R)$ such that there is an exact triangle
\[\canoide_M[d-n]\la \fude R\Hom^\bullet(M,R(r))\la J^\bullet\la \canoide_M[d-n+1].\]
Such a $J^\bullet$ is called the {\it truncated dualizing complex} of $M$.  We observe that $J^\bullet$ has the following properties:
\begin{enumerate}
\item
We have isomorphisms of graded $R$-modules
\[\H^{i}(J^\bullet)\cong H^i(\fude R\Hom^\bullet(M,R(r)))\]
for all $i\neq n-d$. Hence by the local duality theorem for graded $R$-modules (\cite[Theorem 3.6.19]{CMrings}) we have graded isomorphisms
\[\H_\mm^i(M)\iso\H^{n-i}(J^\bullet)^\vee\]
for all $i\not=d$.
\item
Since $M$ is generalized Cohen-Macaulay,  $\H^{n-i}(J^\bullet)^\vee$ has finite length for all $i\in Z$ by property ($1$) and by the fact $\H^{n-d}(J^\bullet)=0$. So $\H^{n-i}(J^\bullet)$ has finite length for all $i\in Z$. Hence    $J^\bullet\iso\fude R\Gamma_\mm(J^\bullet)$ in the derived category $\fff D{}R$ by Remark\til\ref{torsionfunc}.
\end{enumerate}
Since \begin{eqnarray*} \H^{n-i}(J^\bullet)\iso
\begin{cases}
\Ext_R^{n-i}(M, \canoide_R)&\text{ if \ } i \neq d,\\
0&\text{  if \ } i=d,
\end{cases}
\end{eqnarray*}
is a \fg\ $R$-module for each $i\in\Z$ and $R(r)[n]\iso\canoide_R[n]$ is a normalized dualizing complex, by the local duality theorem in the derived category \cite[Thoerem $6.2$]{RaD} we have that 
$\fude R\Gamma_\mm(J^\bullet))\iso \fude R\Hom^\bullet(J^\bullet,R(r)[n])^\vee$
in $\fff D{}R$. Hence, using again the graded local duality theorem, for each $i\neq d$ we have graded isomorphisms of $R$-modules:
\begin{eqnarray*}  
\H_\mm^i(M)^\vee&\iso&H^{n-i}(J^\bullet)\\
&\iso&H^{n-i}(\fude R\Gamma_\mm(J^\bullet))\\
&\iso&H^{n-i}( \fude R\Hom^\bullet(J^\bullet,R(r)[n])^\vee)\\
&\iso&H^{i-n}(\fude R\Hom^\bullet(J^\bullet,R(r))[n])^\vee
\\
&\iso&\Ext^i_R(J^\bullet,R(r))^\vee.
\end{eqnarray*}
Since $\fude R\Hom^\bullet(-,R(r))$ is a $\delta$-functor and $R(r)$ is a dualizing complex, 
\[\xymatrix@=0.4cm{
\fude R\Hom^\bullet(J^\bullet,R(r))\ar[r] &M\ar[r] &\fude R\Hom^\bullet(\canoide_M,R(r))[n-d]\ar[r] &\fude R\Hom^\bullet(J^\bullet,R(r))[1]
}\]
is also an exact triangle in ${\bf D}^b(R)$, and it yields
the exact sequence of graded $R$-modules
\[0\la\Ext^0_R(J^\bullet,R(r))\la M\la\Ext^{n-d}_R(\canoide_M,R(r))\la\Ext^1_R(J^\bullet,R(r))\la 0\]
and the following graded isomorphisms for each $i>0$:
\[\Ext^{n-d+i}_R(\canoide_M,R(r))\iso\Ext^{i+1}_R(J^\bullet,R(r)).\]
Since $\Ext^i_R(J^\bullet,R(r))^\vee\iso\H_\mm^i(M)^\vee$ for all $i\neq d$, we have for each $i>0$ with $i\neq d-1$:
\[\Ext^{n-d+i}_R(\canoide_M,R(r))^\vee\iso\Ext^{i+1}_R(J^\bullet,R(r))^\vee\iso\H_\mm^{i+1}(M)^\vee.\]
Therefore, using again the graded local duality theorem, we have graded isomorphisms of $R$-modules for each $0<i<d-1$
\[\H_\mm^{d-i}(M)^\vee\iso\H_\mm^{i+1}(\canoide_M),\]
and, if $d\geq 2$, we have that
\[0\la \H_\mm^1(M)^\vee \la H_\mm^d(\canoide_M)\la M^\vee\la \H_\mm^0(M)^\vee\la 0\]
 is an exact sequence of graded $R$-modules.
 \end{proof}
 
A different proof of the above result can be found in 
\cite[Theorem 2.3 (iii) and (iv)]{FH}.

\smallskip

The following can be thought as a Poincar\'e duality in commutative algebra (later the connection will be more clear):

\begin{proposition}\label{cPoincare}
With our notation, if $R/\cc$ is generalized Cohen-Macaulay, then for each $i=1,\ldots ,d-2$, there are graded isomorphisms of $R$-modules 
\[\H^{d-i}_\mm(R/\cc)^\vee\iso\H_\mm^{i+1}(R/\cc)(c).\]
\end{proposition}

Two consequences of the above proposition are:

\begin{corollary}\label{cweird}
With our notation, if $R/\cc$ is generalized Cohen-Macaulay and $c>0$, then there are two possibilities:
\begin{enumerate}
\item $R/\cc$ is Cohen-Macaulay, and so Gorenstein.
\item There exists $2\leq i<d$ such that $\H^i_\mm(R/\cc)_j\neq 0$ for some $j>0$.
\end{enumerate}
Analogously, if $c<0$, then there are two possibilities:
\begin{enumerate}
\item $R/\cc$ is Cohen-Macaulay, and so Gorenstein.
\item There exists $2\leq i<d$ such that $\H^i_\mm(R/\cc)_j\neq 0$ for some $j<0$.
\end{enumerate}
\end{corollary}
\begin{proof}
The proof of the second part of the statement is analogous to the first one, so we focus on the proof of the first part, namely when $c>0$.
If $R/\cc$ is not Cohen-Macaulay, since it satisfies $(S_2)$ there exists $1\leq i\leq d-2$ and $j\in\Z$ such that $\H^{i+1}_\mm(R/\cc)_j\neq 0$. If $j>0$ we are done; otherwise by Proposition \ref{cPoincare} $(\H^{d-i}_\mm(R/\cc)^\vee)_{j-c}\cong (\H^{i+1}_\mm(R/\cc))_j\neq 0$, therefore $\H^{d-i}_\mm(R/\cc)_{c-j}\neq 0$ (and $c-j>0$).
\end{proof}

\begin{corollary}\label{qGorsing}
Let $A$ be a positively graded $k$-algebra, where $k$ is a field, such that $A(a)$ is a graded canonical module and $A$ is generalized Cohen-Macaulay but not Cohen-Macaulay.
If $k$ has positive characteristic and $A$ is $F$-injective, then $a=0$.
\end{corollary}
\begin{proof}
We see $A$ as a quotient of a polynomial ring $R=k[x_1,\dots ,x_n]$ modulo a homogeneous ideal and put $d=\dim(A)$.

Since $\H^{d}_\mm(A)_a\neq 0$, we have $a\leq 0$ because $A$ is $F$-injective by \cite[Proposition 2.4(a)]{HoRo}. If $a<0$, then by Corollary \ref{cweird} there exists $2\leq i<d$ such that $\H^i_\mm(A)_j\neq 0$ for some $j\not=0$, and this is a contradiction to the fact that $A$ is $F$-injective and generalized Cohen-Macaulay by \cite[Proposition 2.4(b)]{HoRo}.
\end{proof}


If $R/\aa$ is generalized Cohen-Macaulay, combining Lemma\til\ref{exactseq} and Lemma \ref{gr:Schenzel} we obtain the following:

\begin{theorem}\label{main}
With our notation, if $R/\aa$ is generalized Cohen-Macaulay, there are
 exact sequences of graded $R$-modules:
\[0\la \H_\mm^1(R/\aa)^\vee\la H_\mm^d(\canoide_{R/\aa})\la(R/\aa)^\vee\la 0,\]
and 
\begin{eqnarray*}
0\la \H^{1}_\mm(R/\bb)(c)\la\H_\mm^{d-1}(R/\aa)^\vee\la\cdots
\H_\mm^{3}(R/\aa)^\vee\la \\
\H^{d-2}_\mm(R/\cc)(c)\la
\H^{d-2}_\mm(R/\bb)(c)\la 
\H_\mm^{2}(R/\aa)^\vee\la \\
\H^{d-1}_\mm(R/\cc)(c)\la
\H^{d-1}_\mm(R/\bb)(c)\la \H_\mm^1(R/\aa)^\vee\la 0
\end{eqnarray*}
\end{theorem}
\begin{proof}
Recall we are assuming that $d\geq 2$. In this case, because $\aa$ is unmixed of height $g=n-d<n$, we have $H^0_\mm(R/\aa)=0$. So Lemma \ref{gr:Schenzel} gives us the short exact sequence of graded $R$-modules:
\[0\la \H_\mm^1(R/\aa)^\vee\la H_\mm^d(\canoide_{R/\aa})\la(R/\aa)^\vee\la 0,\]
This also says that the kernel of the surjective map $H_\mm^d(\canoide_{R/\aa})\la(R/\aa)^\vee$ is $\H_\mm^1(R/\aa)^\vee$ and this fact, together with Lemmas \ref{exactseq} and \ref{gr:Schenzel}, gives us the second desired exact sequence of graded $R$-modules:
\begin{eqnarray*}
0\la \H^{1}_\mm(R/\bb)(c)\la\H_\mm^{d-1}(R/\aa)^\vee\la\cdots
\H_\mm^{3}(R/\aa)^\vee\la \\
\H^{d-2}_\mm(R/\cc)(c)\la
\H^{d-2}_\mm(R/\bb)(c)\la 
\H_\mm^{2}(R/\aa)^\vee\la \\
\H^{d-1}_\mm(R/\cc)(c)\la
\H^{d-1}_\mm(R/\bb)(c)\la \H_\mm^1(R/\aa)^\vee\la 0
\end{eqnarray*}
(The $0$ on the left hand side of the above exact sequence follows because $R/\cc$ satisfies Serre's condition $(S_2)$, so $\H^1_{\mm}(R/\cc)=0$).
\end{proof}

\begin{remark}
We record two special cases of the above result.
\begin{itemize}
\item[(i)] If $R/\aa$ is  Cohen-Macaulay, then Theorem \ref{main} tells us that we have graded isomorphisms:
\[H_\mm^{i}(R/\cc)\iso\H_\mm^{i}(R/\bb) \ \ \forall \ i<d.\]
\item[(ii)] If $R/\cc$ is a Gorenstein ring and $R/\aa$ is generalized Cohen-Macaulay, then Theorem \ref{main} tells us that we have graded isomorphisms:
\[\H_\mm^{i}(R/\bb)(c)\iso\H_\mm^{d-i}(R/\aa)^\vee \ \ \forall \ 0<i<d.\]
This recovers Corollaries 3.3 and 5.3 in \cite{schenzel} (notice that in Corollary 5.3 of \cite{schenzel} there is a mistake in the indexes).
\end{itemize}
\end{remark}

\section{Connectedness properties}\label{sec:conn}

In this section we show some applications to some graphs detecting connectedness properties of $\Spec(R/\cc)$.

\begin{definition}
For every $t\in\{0,\ldots ,d\}$, $\Gamma_t(R/\cc)$ denotes the simple graph with vertices the minimal prime ideals of $R/\cc$ and edges $\{\pp,\qq\}$ with $\altezza(\pp+\qq)\leq t$.
\end{definition}

For simplicity, we just write $\Gamma_t$ for $\Gamma_t(R/\cc)$.

\begin{remark}
The graphs $\Gamma_t$ form a sequence of graphs on the vertex set $\Min(\cc)$ such that:
\begin{itemize}
\item[(i)] $\Gamma_0\subseteq \Gamma_1\subseteq \ldots\subseteq \Gamma_d$.
\item[(ii)] $\Gamma_0$ has only isolated vertices.
\item[(iii)] $\Gamma_d$ it the complete graph.
\item[(iv)] When $R_0$ is a field, $\Gamma_1$ is the {\it dual graph} of the projective variety $\proj(R/\cc)$.
\item[(v)] If $\sharp$ stands for the number of connected components, and if $X=\Spec(R/\cc)$, then
\[\sharp \Gamma_t=\max\{\sharp (X\setminus Z):Z\mbox{ is a closed subset of $X$ with }\codim_XZ> t\}.\]
\end{itemize}
Besides $\Gamma_1$, since $X$ is necessarily connected, the graph $\Gamma_{d-1}$ is of particular interest, because in view of (v) it detects the number of connected components of the {\it punctured spectrum} $\Spec(R/\cc)\setminus \{\mm\}$.
\end{remark}

We recall that for a simple graph $G$ on a vertex set $V$, given a subset $U\subseteq V$ the induced graph $G\restriction_U$ is the simple graph on the vertex set $U$ with as edges $\{i,j\}\subseteq U$ such that $\{i,j\}$ is an edge of $G$.

\begin{theorem}\label{t:connectedness}
Assume that $\cc$ is radical and $R/\pp$ is Cohen-Macaulay for any $\pp\in\Min(\cc)$. For any subset of vertices $B\subset \Min(\cc)$ such that the induced graph $\Gamma_2\restriction_B$ consists of isolated vertices, the removal of the vertices in $B$ from $\Gamma_1$ does not disconnect it.
\end{theorem}
\begin{proof}
Set $\bb=\cap_{\pp\in B}\pp$. We claim that $H_{\mm}^{d-1}(R/\bb)=0$: to check this, let us argue inductively writing $B=B'\cup\{\pp\}$ with $\pp\notin B'$, and $\bb'=\cap_{\pp\in B'}\pp$. Consider the short exact sequence
\[0\to R/\bb\to R/\bb'\oplus R/\pp\to R/(\bb'+\pp)\to 0.\]
Taking the long exact sequence of local cohomology we get:
\[\ldots \to H^{d-2}_{\mm}(R/(\bb'+\pp))\to H^{d-1}_{\mm}(R/\bb)\to H^{d-1}_{\mm}(R/\bb')\oplus H^{d-1}_{\mm}(R/\pp)\to \ldots\]
We have $H^{d-1}_{\mm}(R/\pp)=0$ by assumption and $H^{d-1}_{\mm}(R/\bb')=0$ by induction. Furthermore $\altezza(\bb'+\pp)>g+2$ because the sum of any two prime ideals in $B$ has height $>g+2$ by assumption, so $\dim R/(\bb'+\pp)<d-2$: hence $H^{d-2}_{\mm}(R/(\bb'+\pp))=0$, and consequently $H^{d-1}_{\mm}(R/\bb)=0$.

Therefore, if $\aa=\cap_{\pp\in \Min(\cc)\setminus B}\pp$, $\aa$ and $\bb$ are linked via $\cc$ because $\cc$, being a radical ideal, is equal to $\aa\cap \bb$: so $R/\aa$ satisfies Serre's condition $(S_2)$ by Corollary \ref{serre}. Then $\Spec(R/\aa)$ is connected in codimension 1 by \cite[Corollary 2.3]{Har}, namely $\Gamma_1(R/\aa)$, that is $\Gamma_1$ after removing the subset of vertices $B$, is connected.
\end{proof}

Of course any singleton $B$ satisfies the hypothesis of Theorem \ref{t:connectedness} so, recalling that a graph is 2-connected if the removal of any single vertex does not disconnect it, we get:

\begin{corollary}
If $\cc$ is radical and $R/\pp$ is Cohen-Macaulay for any $\pp\in\Min(\cc)$, then $\Gamma_1$ is 2-connected.
\end{corollary}

\begin{remark}
In the proof of Theorem\til\ref{t:connectedness} we just needed $H_{\mm}^{d-1}(R/\pp)=0$ for all $p\in\Min(\cc)$ so, both in Theorem \ref{t:connectedness} and in the Corollary below, we can replace the hypothesis that $R/\pp$ is Cohen-Macaulay for any $\pp\in\Min(\cc)$ by the weaker one that $H_{\mm}^{d-1}(R/\pp)=0$ for any $\pp\in\Min(\cc)$.
\end{remark}

Another interesting consequence on the connectedness properties of $R/\cc$ is the following:

\begin{theorem}\label{t:conn}
Assume that $R_0$ is a field, that $R$ is standard graded, and that $\cc$ defines a subspace arrangement (i.e. $\cc$ is radical and $\pp=(\pp\cap R_1)$ for all $\pp\in\Min(\cc)$). Then, as soon as $\proj(R/\aa)$ is Cohen-Macaulay (equivalently, $R/\aa$ is generalized Cohen-Macaulay), $\proj(R/\bb)$ is connected whenever $|\Min(\aa)|<d+c$.
\end{theorem}
\begin{proof}
Notice that since $\cc$ is radical, then both $\aa$ and $\bb$ are radical as well. So calling $A=\Min(\aa)\subset \Min(\cc)$ and $B=\Min(\bb)=\Min(\cc)\setminus \Min(\aa)$, we have $\aa=\cap_{\pp\in A}\pp$ and $\bb=\cap_{\pp\in B}\pp$. In particular $\aa$ defines a subspace arrangement, and its Castelnuovo-Mumford regularity is at most $|A|$ by \cite[Theorem 2.1]{DS}. Equivalently the Castelnuovo-Mumford regularity of $R/\aa$ is at most $|A|-1<d+c-1$: in particular $H^{d-1}_{\mm}(R/\aa)_c=0$, that is $[H^{d-1}_{\mm}(R/\aa)^\vee]_{-c}=0$. Theorem \ref{main} tells us that $H^{1}_{\mm}(R/\bb)(c)_{-c}=H^{1}_{\mm}(R/\bb)_0=0$, yielding the connectedness of $\proj(R/\bb)$.
\end{proof}

\begin{remark}
Note that, under the assumptions that Theorem \ref{t:conn}, the fact that $\proj(R/\bb)$ is connected is equivalent to the fact that the graph $\Gamma_{d-1}\restriction_B$ is connected.

Moreover, if $d=2$ the condition that $R/\aa$ is generalized Cohen-Macaulay is vacuous, so Theorem \ref{t:conn} implies that the graph $\Gamma_{d-1}=\Gamma_1$ is $(c+2)$-connected when $d=2$: this recovers the main result of \cite{BV}. 
\end{remark}

\section{Applications to triangulated homology manifolds}\label{s:Lefschetz}

Let $R=k[X_1,\ldots ,X_n]$ be a polynomial ring in $n$ variables over a field $k$ equipped with the standard grading $\deg(X_i)=1$ for any $i=1,\ldots ,n$. Notice that $R$ is a graded $n$-dimensional Gorenstein ring with the unique maximal homogeneous ideal $\mm=(X_1,\ldots ,X_n)$, and that the $a$-invariant of $R$ is $-n$. So the theory developed so far applies in this case.

\vspace{2mm}

Recall that, given a simplicial complex $\Delta$ on $n$ vertices, the {\it Stanley-Reisner ideal} $I_{\Delta}\subseteq R$ is the ideal generated by the monomials $X_{i_1}X_{i_2}\cdots X_{i_s}$ such that $\{i_1,\ldots ,i_s\}$ is not a face of $\Delta$. This association yields a 1 to 1 correspondence between simplicial complexes on $n$ vertices and squarefree monomial ideals of $R$, and $k[\Delta]=R/I_{\Delta}$ is called the {\it Stanley-Reisner} ring of $\Delta$. Let us remind that, given a face $\sigma$ of a simplicial complex $\Delta$, by definition the dimension of $\sigma$ is $\dim\sigma=|\sigma | -1$ and the dimension of $\Delta$ is $\dim\Delta=\max\{\dim\sigma:\sigma\in\Delta\}$. We say that $\sigma$ is an {\it $i$-face} if $\dim\sigma=i$.
Given a face $\sigma\in\Delta$, its {\it link} is the simplicial complex defined as $\lk_{\Delta}\sigma=\{\tau\in\Delta:\tau\cup\sigma\in\Delta \mbox{ and }\tau\cap\sigma=\emptyset\}$.
A face of $\Delta$ maximal by inclusion is called a {\it facet}, and $\Delta$ is a pure $d$-dimensional simplicial complex if $\dim\sigma=d$ for any facet $\sigma\in\Delta$.

We say that a pure $d$-dimensional simplicial complex $\Delta$ is a {\it normal pseudomanifold} if the following two conditions hold:
\begin{enumerate}
\item $\lk_{\Delta}(\sigma)$ is connected for any $i$-face $\sigma$ with $i\leq d-2$.
\item Any $(d-1)$-face is contained in exactly two $d$-faces.
\end{enumerate}

\begin{remark}
Gr\"abe \cite{Grabe} studied the notion of {\it quasi $d$-manifold}, that differs from that of $d$-dimensional normal pseudomanifold only because (1) is required for {\it nonempty} $i$-faces with $i\leq d-2$ and in (2) the word ``exactly" is replaced by ``at most". 
\end{remark}

\begin{remark}
Condition (1) above can be defined for any simplicial complex, and is usually referred as {\it normality}. It is not difficult to check that a simplicial complex $\Delta$ is normal if and only if $k[\Delta]$ satisfies Serre's condition $(S_2)$ (independently on the field $k$). This, somehow, resembles Serre's criterion for normality: a Noetherian ring $A$ is normal if and only if $A$ satisfies $(R_1)$ and $(S_2)$ (notice that a Stanley-Reisner ring $k[\Delta]$ almost never satisfies $(R_1)$).

Usually, in the definition of pseudomanifold condition (1) is replaced by the weaker one that $\Delta$ is strongly connected (namely, one can walk along the facets through codimension one faces). While the links of a pseudomanifold are not necessarily pseudomanifolds, the links of a normal pseudomanifold are themselves normal pseudomanifolds.  
\end{remark}

\begin{remark}\label{r:orientable}
It is not difficult to check that a $d$-dimensional (normal) pseudomanifold $\Delta$ is orientable if and only if $H_d(\Delta;\Z)\neq 0$, that is furthermore equivalent to the fact that $H_d(\Delta;\Z)\cong \Z$. Indeed, the summation over the facets $\sigma\in \Delta$ of $m\sgn(\sigma)\sigma$ gives a top-dimensional (and so nontrivial) cycle for all $m\in\Z$. On the other hand, given a top dimensional cycle $\sum_{\sigma}a_{\sigma}\sigma$ with $a_{\sigma}\in\Z$, using condition (2) in the definition of normal pseudomanifolds, and since $\Delta$ is strongly connected, we must have $a_{\sigma}\neq 0$ for any facet $\sigma$. Moreover, since $\Delta$ is strongly connected and using condition (2) in the definition of normal pseudomanifolds, there must exist a natural number $m\in\N$ such that $a_{\sigma}=\pm m$ for every facet $\sigma$. It is therefore enough to define accordingly $\sgn(\sigma)$.

Notice that, if $k$ has characteristic 2 and $\Delta$ is a $d$-dimensional (normal) pseudomanifold, then $H_d(\Delta;k)\neq 0$, since $\sum_{\sigma}\sigma$ is always a top-dimensional cycle (where the summation runs over all the facets). On the other hand, there are non-orientable normal pseudomanifolds, for example a triangulation of the real projective plane.
\end{remark}

\begin{theorem}\label{t:qGpseudo}
Given a simplicial complex  $\Delta$  on $n$ vertices, the following are equivalent:
\begin{enumerate}
    \item There is a graded isomorphism of $k[\Delta]$-modules $\omega_{k[\Delta]}\cong k[\Delta]$.
    \item $\Delta$ is a normal pseudomanifold and $H_{\dim\Delta}(\Delta;k)\neq 0$.
\end{enumerate}
\end{theorem}
\begin{proof}
(2) $\implies$ (1) follows by the theorem at page 169 of the paper of Gr\"abe \cite{Grabe}:  indeed a normal pseudomanifold is a quasi-manifold. Moreover it is easy to see that the condition $H_{\dim\Delta}(\Delta;k)\neq 0$ implies, using the notations of \cite{Grabe}, $\mathrm{Bd}(\Delta)=\emptyset$ (see \cite[page 167, before Beispiele 3.3]{Grabe}). So, still using the notation in Gr\"abe's paper, $J(\Delta)=k[\Delta]$ and $H_{\dim\Delta}(\Delta,\mathrm{Bd}(\Delta);k)=H_{\dim\Delta}(\Delta;k)\neq 0$, hence we get the desired statement.

(1) $\implies$ (2): Since $k[\Delta]\cong \omega_{k[\Delta]}$, $k[\Delta]$ satisfies the $(S_2)$ condition, and so $\Delta$ must be normal. Furthermore, let $\dim\Delta=d$, and take a $(d-1)$-face $\sigma$. Then $\lk_{\Delta}\sigma$ is a $0$-dimensional complex consisting of $r$ vertices, where $r$ is the number of facets containing $\sigma$. Since the quasi-Gorenstein property localizes, and since $k[\Delta]_{X_{\sigma}}\cong k[\lk_{\Delta}\sigma][X_i,X_i^{-1}:i\in\sigma]$ (where $X_{\sigma}=\prod_{i\in\sigma}X_i$), it turns out that $k[\lk_{\Delta}\sigma]$ is a 1-dimensional quasi-Gorenstein (and so Gorenstein) ring. Since $\reg k[\lk_{\Delta}\sigma]\leq \dim k[\lk_{\Delta}\sigma]=1$, either $r=1$ or $r=2$, since the only standard graded Gorenstein algebras of Castelnuovo-Mumford regularity 1 are hypersurfaces; so $\Delta$ is a quasi-manifold in the sense of Gr\"abe. We want to prove that $r=1$ is not possible: notice that the graded isomorphism $\omega_{k[\Delta]}\cong k[\Delta]$ implies that $H_d(\Delta;k)^\vee=H^{d+1}_{\mm}(k[\Delta])_0=k\neq 0$. Since $H_d(\Delta;k)\neq 0$, we have $\widetilde{H_0}(\lk_{\Delta}\sigma;k)\neq 0$ by \cite[page 167, before Beispiele 3.3]{Grabe} and the lemma at page 162 of \cite{Grabe}. So $r=2$ is the only possibility, and we conclude the proof since we proved along the way that $H_d(\Delta;k)\neq 0$.
\end{proof}

We recall that a simplicial complex $\Delta$ is {\it Buchsbaum}  over $k$ if, for all $\sigma\neq \emptyset$, $\tilde{H}_i(\lk_{\Delta}\sigma;k)=0$ for all $i<\dim \lk_{\Delta}\sigma$. This condition is equivalent to the fact that $k[\Delta]$ is generalized Cohen-Macaulay. Furthermore, $\Delta$ is a {\it homology manifold} over $k$ if, besides being Buchsbaum, $H_{\dim \lk_{\Delta}\sigma}(\lk_{\Delta}\sigma;k)\cong k$ for any nonempty face $\sigma\in\Delta$. It is a {\it homology sphere} over $k$ if, besides being a homology manifold, $\widetilde{H_j}(\Delta;k)=0$ if $j<\dim\Delta$ and $H_{\dim \Delta}(\Delta;k)\cong k$.

\begin{corollary}\label{c:manifold}
Let $\Delta$ be a connected Buchsbaum simplicial complex over the field $k$. Then the following are equivalent:
\begin{enumerate}
    \item $k[\Delta]$ is quasi-Gorenstein but not Gorenstein;
    \item $\Delta$ is a homology manifold, but not a homology sphere, over $k$ and $H_{\dim\Delta}(\Delta;k)\neq 0$.
    \end{enumerate}
\end{corollary}
\begin{proof}
(2) $\implies$ (1) follows from Theorem \ref{t:qGpseudo}: indeed a connected homology manifold $\Delta$ is in particular a normal pseudomanifold, and $H_{\dim\Delta}(\Delta;k)$ is (dual to) the degree 0 part of $\omega_{k[\Delta]}$, that is $k$ by Theorem \ref{t:qGpseudo} because $H_{\dim\Delta}(\Delta;k)\neq 0$. Since $\Delta$ is not a homology sphere, therefore $H_i(\Delta;k)\neq 0$ for some $i<\dim\Delta$, so $k[\Delta]$ is not Cohen-Macaulay. 

For (1) $\implies$ (2), notice that the $a$-invariant of a Stanley-Reisner ring cannot be positive. If the $a$-invariant of $k[\Delta]$ were negative, since $k[\Delta]$ is generalized Cohen-Macaulay, by Corollary \ref{cweird} there would exist $2\leq i<d$ such that $\H^i_\mm(k[\Delta])_j\neq 0$ for some $j<0$. This is impossible for a Buchsbaum simplicial complex, so we must have that $k[\Delta]$ is a quasi-Gorenstein ring of $a$-invariant 0. Namely, there is a graded isomorphism of $k[\Delta]$-modules $\omega_{k[\Delta]}\cong k[\Delta]$. So $\Delta$ is a normal pseudomanifold and $H_{\dim\Delta}(\Delta;k)\neq 0$ by Theorem \ref{t:qGpseudo}. Since a Buchsbaum normal pseudomanifold is a connected homology manifold (e.g. see \cite[Beispiele 3.3 (2)]{Grabe}, we conclude.
\end{proof}

\begin{remark}
If $\Delta$ has dimension 2 and $k$ has characteristic 0, Corollary \ref{c:manifold} tells that $k[\Delta]$ is quasi-Gorenstein but not Gorenstein if and only if $\Delta$ is an orientable manifold but not a sphere. Indeed since $k[\Delta]$ satisfies $(S_2)$ and has Krull dimension 3, it is necessarily generalized Cohen-Macaulay, so we are allowed to use Corollary \ref{c:manifold}. Furthermore in dimension 2 homology manifolds (respectively homology spheres) coincide with manifolds (respectlvely spheres). Finally, since $k$ is flat as $\Z$-module, for any simplicial complex $\Gamma$ we have $H_i(\Gamma;k)\cong H_i(\Gamma;\Z)\otimes_{\Z}k$ for every integer $i$, hence $H_{\dim\Gamma}(\Gamma;\Z)\neq 0$ if and only if $H_{\dim\Gamma}(\Gamma;k)\neq 0$ since $H_{\dim\Gamma}(\Gamma;\Z)$ is torsionfree.
\end{remark}

\begin{theorem}\label{Lefschetz duality}
Let $\Delta$ be a $d$-dimensional normal pseudomanifold such that $\widetilde{H}^d(\Delta;k)\neq 0$ (e.g. if $\Delta$ is orientable) with facets $\sigma_1,\ldots ,\sigma_l$. Let $A,B$ be a partition of $\{1,\ldots ,l\}$, $\Delta_A$ be the simplicial complex with facets $\sigma_i$ where $i\in A$ and $\Delta_B$ be the simplicial complex with facets $\sigma_i$ where $i\in B$. If $\Delta_A$ is Buchsbaum, then there is a long exact sequence
\begin{eqnarray*}
0\to \widetilde{H}^0(\Delta_B;k)\to \widetilde{H}_{d-1}(\Delta_A;k)\to \widetilde{H}^1(\Delta;k)\to \widetilde{H}^1(\Delta_B;k)\to 
\widetilde{H}_{d-2}(\Delta_A;k)\to \\
\widetilde{H}^2(\Delta;k) \to \cdots \to \widetilde{H}^{d-2}(\Delta_B;k) \to \widetilde{H}_{1}(\Delta_A;k)\to \widetilde{H}^{d-1}(\Delta;k)\to \widetilde{H}^{d-1}(\Delta_B;k)\\
\to\widetilde{H}_0(\Delta_A;k)\to 0.
\end{eqnarray*}
Moreover $\widetilde{H}^{i}(\lk_{\Delta_B}\sigma;k)\cong \widetilde{H}^{i}(\lk_{\Delta}\sigma;k)$ for all $\sigma\neq \emptyset$ and $i<d$. In particular $\widetilde{H}^{i}(\lk_{\Delta}\sigma;k)=0$ for any $\sigma\notin \Delta_B$ and $i<d$.
\end{theorem}
\begin{proof}
Let $\cc=I_{\Delta}$, $\aa=I_{\Delta_A}$ and $\bb=I_{\Delta_B}$. Because these ideals are radical and $\cc=\aa\cap \bb$, we have $\aa:_R\cc=\bb$ and $\bb:_R\cc=\aa$. Moreover $\omega_{R/\cc}\cong R/\cc$ by Theorem \ref{t:qGpseudo}. Furthermore the assumption that $\Delta_A$ is Buchsbaum is equivalent to say that $R/\aa$ is generalized Cohen-Macaulay, hence we can apply Theorem \ref{main} with $c=0$. Now the existence of the long exact sequence follows by Hochster's formula (for example see \cite[Theorem 5.3.8]{CMrings}). Also the last part of the thesis follows in the same way, just noticing that the maps of the long exact sequence in Theorem \ref{main} preserve the $\Z^n$-degree.
\end{proof}

The above theorem is a version of the classical Lefschetz duality in algebraic topology (c.f. \cite[Theorem 70.2]{ETA}). The point is that, with the notation of Theorem \ref{Lefschetz duality}, we also have the classical long exact sequence of relative singular cohomology
\begin{eqnarray*}
0\to \widetilde{H}^0(\Delta_B;k)\to \widetilde{H}^1(\Delta,\Delta_B;k)\to \widetilde{H}^1(\Delta;k)\to \widetilde{H}^1(\Delta_B;k)\to 
\widetilde{H}^2(\Delta,\Delta_B;k)\to \\
\cdots \to \widetilde{H}^{d-2}(\Delta_B;k) \to \widetilde{H}^{d-1}(\Delta,\Delta_B;k)\to \widetilde{H}^{d-1}(\Delta;k)\to \widetilde{H}^{d-1}(\Delta_B;k).
\end{eqnarray*}
Therefore Theorem \ref{Lefschetz duality} gives commutative diagrams for all $i=1,\ldots ,d-1$:
\[\xymatrix@=0.7cm{
\widetilde{H}^{i-1}(\Delta;k)\ar[d]_=\ar[r]&\widetilde{H}^{i-1}(\Delta_B;k)\ar[d]_=\ar[r]&\widetilde{H}^{i}(\Delta,\Delta_B;k)\ar[r]&\widetilde{H}^{i}(\Delta;k)\ar[d]_=\ar[r]&\widetilde{H}^{i}(\Delta_B;k)\ar[d]_=\\
\widetilde{H}^{i-1}(\Delta;k)\ar[r]&\widetilde{H}^{i-1}(\Delta_B;k)\ar[r]&\widetilde{H}_{d-i}(\Delta_A;k)\ar[r]&\widetilde{H}^{i}(\Delta;k)\ar[r]&\widetilde{H}^{i}(\Delta_B;k)}.\]
Then there are also vertical maps $\widetilde{H}^{i}(\Delta,\Delta_B;k)\xrightarrow{f_i} \widetilde{H}_{d-i}(\Delta_A;k)$ making all squares commute, and by the five lemma $f_i$ has to be an isomorphism for all $i=1,\ldots ,d-2$. Since $\widetilde{H}_{d-i}(\Delta_A;k)\cong \widetilde{H}_{d-i}(|\Delta|\setminus|\Delta_B|;k)$ by the lemma below, we get the topological Lefschetz duality \cite[Theorem 70.2]{ETA} when the coefficient group is $k$.

\begin{lemma}\label{l:homotopy}
Let $\Delta$ be a pure $d$-dimensional simplicial complex with facets $\sigma_1,\ldots ,\sigma_k$ such that every $(d-1)$-face is contained in at most two facets and $\lk_{\Delta}(\sigma)$ is connected for any $\emptyset\neq \sigma\in \Delta$ of dimension $\leq d-2$ (e.g. if $\Delta$ is a $d$-dimensional normal pseudomanifold). Let $A,B$ be a partition of $\{1,\ldots ,k\}$, $\Delta_A$ be the simplicial complex with facets $\sigma_i$ where $i\in A$ and $\Delta_B$ be the simplicial complex with facets $\sigma_i$ where $i\in B$. If $\lk_{\Delta_A}(\sigma)$ is connected for any $\emptyset\neq \sigma\in \Delta_A$ of dimension $\leq d-2$ (e.g. if $\Delta_A$ is Buchsbaum), the topological spaces $|\Delta_A|$ and $|\Delta|\setminus|\Delta_B|$ are homotopically equivalent.
\end{lemma}
\begin{proof}
Let $\Gamma$ be the subcomplex of $\Delta_A$ consisting of all faces disjoint from $\Delta_B$. Then $|\Gamma|$ is a deformation retract of $|\Delta|\setminus|\Delta_B|$ by \cite[Lemma 70.1]{ETA}. We will show that $\Delta_A$ collapses to $\Gamma$, and this will let us conclude.

To this purpose, consider a vertex $v\in\Delta_A\cap\Delta_B$. Then $\lk_{\Delta_A}v$ is a connected $(d-1)$-dimensional simplicial complex such that each $(d-2)$-face is contained in at most two $(d-1)$-faces. The assumption that $\lk_{\Delta_A}(\sigma)$ is connected for any $\emptyset\neq \sigma\in \Delta_A$ of dimension $\leq d-2$ actually ensures that $\lk_{\Delta_A}v$ is strongly connected, and in particular pure. For the same reason, $\lk_{\Delta}v$ is also strongly connected, so if $\sigma$ is a facet of $\Delta_A\subset \Delta$ containing $v$ and $\tau$ is a facet of $\Delta_B\subset \Delta$ containing $v$, it must be possible to walk from $\sigma$ to $\tau$ through $(d-1)$-faces of $\Delta$ containing $v$. Because each of these $(d-1)$-faces is contained in at most two facets of $\Delta$, there is one $(d-1)$-face $\alpha$ of $\Delta$ containing $v$ which is contained in only one facet of $\Delta_A$. In other words, $\alpha\setminus \{v\}$ is a free $(d-2)$-face of $\lk_{\Delta_A}v$, which therefore is a pure strongly connected $(d-1)$-dimensional simplicial complex such that each $(d-2)$-face is contained in at most two facets, having a free $(d-2)$-face. It is not difficult to show that such a simplicial complex has to be collapsible. One can easily choose to collapse to a vertex $v'\in\lk_{\Delta_A}v$ that is not a vertex of $\Delta_B$. The collapse happened in $\lk_{\Delta_A}v$, but can be done at the level of $\Delta_A$: indeed if $\beta$ is a face of $\lk_{\Delta_A}v$ contained only in one higher dimensional face $\gamma$ of $\lk_{\Delta_A}v$, then $\beta\cup \{v\}$ is a face of $\Delta_A$ contained only in one higher dimensional face $\gamma\cup \{v\}$ of $\Delta_A$ (here $\Delta_A$ should be thought as modifying while collapsing). 

At this point, $\Delta_A$ collapsed to the simplicial complex $(\Delta_A\setminus\{\sigma\cup \{v\}:\sigma\in\lk_\Delta v\})\cup \{\{v\},\{v,v'\}\}$, and we can  still collapse the pairs $(\{v\},\{v,v'\})$. If this new simplicial complex is equal to $\Gamma$ we finished, otherwise renaming it by $\Delta_A$ there exists a vertex $w\in\Delta_A\cap\Delta_B$ and we proceede as above. Since there are finitely many vertices in $\Delta_A\cap\Delta_B$, at a certain point the algorithm will stop.
\end{proof}

\begin{remark}
In the lemma above the assumptions are necessary, as shown by the examples (in the first $\Delta$ is not a manifold and in the second $\Delta_A$ is not Buchsbaum).
\begin{eqnarray*}
\Delta=\langle\{1,2,3\}, \{1,2,4\},\{1,2,5\}\rangle, \ \ \Delta_A=\langle\{1,2,3\}, \{1,2,4\}\rangle , \\
\Delta=\langle\{1,2,5\}, \{2,3,5\},\{3,4,5\},\{1,4,5\},\{1,2,3\}\rangle, \ \ \Delta_A=\langle\{1,2,5\},\{3,4,5\}\rangle .
\end{eqnarray*}
We already explained the connection between Theorem \ref{Lefschetz duality} and the topological Lefschetz duality, but we want also to quickly explain the similarities in the assumptions. It is not difficult to show that the concept of (triangulated) connected homology manifold without boundary is equivalent with the concept of Buchsbaum normal pseudomanifold (see also \cite[Beispiele 3.3 (2)]{Grabe}). This allows to show that the assumptions that $\Delta$ is a normal pseudomanifold and that $\Delta_A$ is Buchsbaum imply that $|\Delta|\setminus |\Delta_B|$ is a homology manifold without boundary (perhaps disconnected).
\end{remark}

\end{document}